\documentclass[conference, 10pt]{IEEEtran}

\usepackage{amsmath, amsthm, amssymb}
\usepackage{latexsym}
\usepackage{graphicx}
\usepackage{verbatim}
\usepackage{mathrsfs}
\usepackage{float}
\usepackage[lofdepth, lotdepth]{subfig}
\usepackage{array}
\usepackage{url}
\usepackage{algorithm}
\usepackage[noend]{algorithmic}
\usepackage[table]{xcolor}
\usepackage{enumerate}
\usepackage{eucal}
\usepackage{stmaryrd}
\usepackage{mathtools}
\usepackage{pgf}
\usepackage{tikz}
\usetikzlibrary{arrows,automata,positioning}
\usepackage[latin1]{inputenc}
\usetikzlibrary{automata}
\usepackage{subfig}

\usepackage[top=0.53in, bottom=0.85in, left=0.55in, right=0.55in]{geometry}
\usepackage{mathtools}

\newcommand{\floor}[1]{\left\lfloor #1 \right\rfloor}

\theoremstyle{plain}
\newtheorem{theorem}{Theorem}
\newtheorem{proposition}[theorem]{Proposition}
\newtheorem{lemma}[theorem]{Lemma}

\newtheorem{corollary}[theorem]{Corollary}

\theoremstyle{definition}
\newtheorem{definition}{Definition}

\newtheorem{remark}[definition]{Remark}

\newcommand{\cS}{{\mathcal S}}

\DeclareMathAlphabet{\mathbfsl}{OT1}{ppl}{b}{it} 

\newcommand{\bu}{{\mathbfsl u}}
\newcommand{\bv}{{\mathbfsl v}}

\newcommand{\bs}{{\mathbfsl s}}

\newcommand{\bw}{{\mathbfsl{w}}}

\newcommand{\bx}{{\mathbfsl{x}}}

\newcommand{\bbZ}{{\mathbb Z}}

\renewcommand{\ge}{\geqslant}
\renewcommand{\le}{\leqslant}

\newcommand{\bd}{\mathtt{d}}

\newcommand{\HH}{\mathbb{H}}

\newcommand{\bCap}{{\rm Cap}}
\newcommand{\vz}{\mathbfsl{z}}
\newcommand{\vk}{\mathbfsl{k}}
\newcommand{\vzero}{\mathbf{0}}
\newcommand{\Hessian}{\mathtt{H}}
\newcommand{\bbS}{\mathbb{S}}
\newcommand{\cT}{\mathcal{T}}
\newcommand{\A}{\mathtt{A}}
\newcommand{\T}{\mathtt{T}}
\newcommand{\C}{\mathtt{C}}
\newcommand{\G}{\mathtt{G}}

\begin{document}
	
\title{Evaluation of the Gilbert--Varshamov Bound using Multivariate Analytic Combinatorics\\[-3mm]}
\author{
		\IEEEauthorblockN{
			Goyal Keshav\IEEEauthorrefmark{1},
			Duc Tu Dao\IEEEauthorrefmark{1},
			Han Mao Kiah\IEEEauthorrefmark{1},
			and
			Mladen Kova{\v{c}}evi{\'c}\IEEEauthorrefmark{2}}
		
		\IEEEauthorblockA{
			\IEEEauthorrefmark{1}\small School of Physical and Mathematical Sciences, Nanyang Technological University, Singapore\\
			\IEEEauthorrefmark{2}\small Department of Electrical Engineering, University of Novi Sad, Serbia\\
			email: \{{keshav002,daoductu001}, {hmkiah}\}@ntu.edu.sg, kmladen@uns.ac.rs\\[-5mm]
		} 
	}
	
	\maketitle
	\hspace{-3mm}\begin{abstract}
		Analytic combinatorics in several variables refers to a suite of tools that provide sharp asymptotic estimates for certain combinatorial quantities.
		In this paper, we apply these tools to determine the Gilbert--Varshamov (GV) bound
		for the sticky insertion and the constrained-synthesis channel.
	\end{abstract}
	

\section{Introduction}\label{sec:intro}

Established in the 1950s, the Gilbert-Varshamov bound \cite{Gilbert1952, Varshamov1957} is a fundamental lower bound on the size of the largest code. 
In this paper, we study the sticky-insertion channel with $L_1$ metric and the constrained-synthesis channel with Hamming metric.
To determine the GV bound, one requires two quantities: the size of the input space, $\cS$, and also, the ball volume, that is, the number of words with distance at most $d - 1$ from a center word. 
Then the GV bound is given by the ratio of $|\cS|$ and the average ball volume \cite{gu1993generalized} (details will be discussed in Section~\ref{sec:GV}).
In~\cite{kolesnik1991generating}, the authors showed that the asymptotic rate of average ball volume can be computed via some optimization problem. Later, Marcus and Roth modified the optimization problem by including an additional constraint and variable, 
and the resulting bound improves the usual GV bound \cite{marcus1992improved}.
In \cite{keshav2022evaluating}, efficient numerical procedures to solve these optimization problems have been provided.

In this work, we propose a different approach to estimate the average ball volume by using multivariate analytic combinatorics (see \cite{pemantle2008} for a survey of combinatorial applications and also,  \cite{Melczer2021} for an introductory text). We remark that the use of generating functions in determining GV bound (and more generally, coding theory) is not new. 
In one of the pioneering papers, Kolesnik and Krachkovsky \cite{kolesnik1991generating} employed generating functions to compute the GV bound for runlength-limited codes.
Recently, new tools were developed in multivariate analytic combinatorics \cite{pemantle2008}. These tools were then used to determine certain asymptotic properties of runlength-limited sequences in~\cite{kovavcevic2019RLL,kovavcevic2021asymptotic} and the capacities of certain cost-constrained channels for DNA synthesis~\cite{lenz2021multivariate}.

\section{Preliminaries}
Let $\Sigma$ be an alphabet, 
$\Sigma^n$ the set of all words of length $n$ over $\Sigma$, and 
$\Sigma^*$ the set of all finite-length words over $\Sigma$. We recall the entropy function $\HH(p) \triangleq -p \log(p) -(1-p)\log(1-p)$\,. The $\log$ notation denotes logarithm base 2. 

\subsection{Gilbert-Varshamov Bound}\label{sec:GV}

Let $\cS\subseteq \Sigma^*$ and set $\cS_n=\cS\cap \Sigma^n$.
Let $\bd: \cS \times \cS \rightarrow \bbZ_{\geq 0}\cup\{\infty\}$ be a metric defined on $\cS$. 
For each $n$, we want to find $C \subseteq \cS_n$, called $(\cS_n,d)$-code, such that $\bd(c_1, c_2) \geq d$ for all $c_1, c_2 \in C$. The largest code size, $A(\cS_n,d) \triangleq \max\{|C|: C \subseteq \cS_n,
\bd(c_1,c_2) \geq\,d, \text{ for all } c_1, c_2 \in C, c_1 \neq c_2\}$, is our quantity of interest.
In terms of asymptotic rates, fixing $0\leq \delta \leq 1$, we aim to find the highest attainable rate
~$\alpha_{\cS}(\delta)=\limsup_{n \to \infty} \frac{\log A(\cS_n,\floor{\delta n})}{n}$.

Let $\bu\in \cS_n$ and define $V(\bu,r) = \{\bv \in \cS_n: \bd(\bu,\bv) \leq r\}$ to be the ball of radius $r$ centered at $\bu$. 
If $|V(\bu,r)|$ is constant over all $\bu \in \cS_n$, the GV bound states that $A(\cS_n,d) \geq |\cS_n|/|V(\bu,d-1)|$. 
Otherwise, the bound needs to be adapted. 
Kolesnik and Krachkovsky~\cite{kolesnik1991generating}showed that 
the GV lower bound can be generalized to $|\cS|/4\overline{V(d-1)}$ where $\overline{V(d-1)} = \frac{1}{|\cS_n|}\sum_{\bu\in\cS_n} |V(\bu,d-1)| $ is the average ball volume. 
This was further improved by Gu and Fuja~\cite{gu1993generalized} to $|\cS_n|/\overline{V(d-1)}$. 
For simplicity, we consider the collection of word pairs
$T(\cS_n,d-1) \triangleq \{(\bu,\bv)\in\cS_n^2 : \bv\in V(\bu,d-1)\}$. 
Hence, $|T(\cS_n,d-1)|$ provides the total ball size and the above result is restated as $A(\cS_n,d)\geq |\cS_n|^2/|T(\cS_n,d-1)|$.

In terms of asymptotic rates, the GV bound asserts that there exists a family of $(\cS_n,\floor{\delta n})$-codes such that their rates approach 
\begin{equation}\label{eq:gv}
	R_{GV}(\cS,\delta) = 2\bCap(\cS) - \widetilde{T}(\cS,\delta)
\end{equation}
where $\bCap(\cS) \triangleq \limsup_{n \rightarrow \infty} \frac{\log |\cS_n|}{n}$, and $\widetilde{T}(\cS,\delta) \triangleq \limsup_{n\rightarrow \infty} \frac{\log |T(\cS_n,\floor{\delta n})|}{n}$. Note that $\bCap(\cS) = \widetilde{T}(\cS,0)$.

In summary, to find a lower bound for the highest achievable rate $R_{GV}(\cS,\delta)$, we need to compute  $\widetilde{T}(\cS,\delta)$.
In the following sections, the set $\cS_n$ will be characterized by some parameters, so we will replace $\cS_n$ with those parameters. The floor function may be omitted for simplicity.

\subsection{Analytic Combinatorics in Several Variables (ACSV)}\label{ACSV}

Finding the total ball volume $|T(\cS_n,\floor{\delta n})|$ or its asymptotic rate $\widetilde{T}(\cS,\delta)$ is the main goal of this paper. 
In many cases, generating functions provide a concise description of  $|T(\cS_n,\floor{\delta n})|$. 
As most of these generating functions involve more than one variable, 
we borrow tools from multivariate analytic combinatorics to provide asymptotic estimates.

Let the number of variables be $\ell$ and let $\vz$ denote the $\ell$-tuple $(z_1,\ldots, z_\ell)$. With $\vk \in \bbZ_{\geq 0}^\ell$, let $\vz^\vk$ denote the monomial $\prod_{i=1}^\ell z_i^{k_i}$. 
Suppose that we have a multivariate array $\{a_{\vk} \in \bbZ_{\geq 0} : \vk \in \bbZ_{\geq 0}^\ell\}$ with the generating function $F(\vz) = \sum_{\vk}a_{\vk} \vz^\vk$. 
The following theorem is crucial for this paper.

\begin{theorem}[Theorem 1.3 in \cite{pemantle2008}]\label{thm:acsv}
Given $F(\vz) = \sum_{\vk}a_{\vk} \vz^\vk= \frac{G(\vz)}{H(\vz)}$ where $G$ and $H$ are both analytic, $H(\vzero) \neq 0$, and $a_{\vk} > 0$.

	For each $\vk= (k_1,k_2,\ldots k_\ell)>\vzero$,  
	there is a unique solution $\vz^* = (z_1^*, z_2^*, \ldots, z_\ell^*) > \vzero$ 
	satisfying the equations
	\begin{equation}\label{eq:partial}
	\begin{split}
		H(\vz) &= 0\\
		k_\ell z_j \frac {\partial H(\vz)}{\partial z_j}&= k_j z_\ell \frac{\partial H(\vz)}{\partial z_\ell} \text{ for } 1 \leq j \leq \ell -1.
	\end{split}
	\end{equation}
	Furthermore, if $G(\vz^*) \neq 0$,
	\begin{equation}\label{eq:approx}
		a_\vk \sim (2\pi)^{-(\ell-1)/2}\Hessian^{-1/2}\frac{G(\vz^*)}{-z_\ell \frac{\partial H(\vz)}{ \partial z_\ell} \large\rvert_{\vz =\vz^*}} k_\ell^{-(\ell-1)/2} (\vz^*)^{-\vk},
	\end{equation}
	\noindent {\small where $\Hessian$ is the determinant of the Hessian of the function parametrizing the hypersurface $\{H=0\}$ in logarithmic coordinates.}
\end{theorem} 
For a detailed calculation of the Hessian matrix, we refer readers to Lemma 5 in \cite{Melczer2021}. 
More general asymptotic results are available in Theorems 5.1--5.4 of \cite{Melczer2021}. For this paper, we consider the case where all coordinates of $\vk$ grow linearly with $n$\, 
i.e. $k_i = nr_i$ where $r_i$ is fixed for $1\leq i\leq \ell$. 
Hence, all terms in \eqref{eq:approx} tend to constants except $k_\ell^{-(\ell-1)/2} (\vz^*)^{-\vk}$. 
Therefore, we simplify the asymptotic behavior of sequence $a_\vk$ as follows:
\begin{align}
	&a_{nr_1,nr_2,\ldots, nr_\ell} = \Theta\left( n^{(\ell-1)/2}\prod_{i=1}^{\ell} (z_i^*)^{-nr_i}\right) \label{eq:asym_approx-0}\\
	&\lim_{n \rightarrow \infty} \frac{\log a_{nr_1, nr_2, \ldots, nr_\ell}}{n}  = -\sum_{i=1}^{\ell} r_i\log z_i^* \label{eq:asym_approx}
\end{align}

\subsection{Our Contribution}

In this work, we apply Theorem~\ref{thm:acsv} to obtain GV bounds for sticky-insertion and constrained-synthesis channels.
Specifically, for each channel, we derive the corresponding multivariate generating functions for the total ball size and 
then set up the corresponding set of equations~\eqref{eq:partial}.
After which, we solve this system and hence, obtain the corresponding GV bounds.

For the sticky-insertion channels, we obtain lower bounds on the rates of length-$n$ binary codes correcting $b$ sticky insertions,
where $b$ is proportional to $n$. 
Previously, such results were only in the instance where $b$ is constant with respect to $n$.

To the best of our knowledge, previous work on codes for DNA synthesis has not studied error correction.
Hence, our work provides a rudimentary lower bound for these codes.

\section{GV Bound for the Sticky-Insertion Channel}\label{sec:Sticky}
This section describes the procedure to compute the GV bound for the {\em sticky-insertion} channel. In this section, $\Sigma = \{0,1\}$.

Formally, for the sticky-insertion channel, the inputs are binary strings of length $n$. The channel acts on the transmitted strings by introducing duplication errors in succession, where duplication is defined as the insertion of the same bit next to its original position.
For example,  $\bu = 1010$ is an input and $\bv = 1\underline{11}010\underline{00}$ is a possible output. Here, the inserted bits are underlined and the total number of errors is four.

It is clear that the channel does not alter the number of runs in the input. 
Hence, we consider the constrained space $\cS(n,r)$ that comprises all binary words of length $n$ with $r$ runs. 
Since the number of runs is preserved, it is more convenient to use  the following quantity
$\bbS(n,r) = \{(u_1,u_2, \ldots, u_r)  \in \mathbb{Z}_{\geq1}^n; u_1 + u_2 + \ldots + u_r = n\}$. In particular, $u_i$ is the length of the $i$-th run. The size of $\bbS(n,r)$ is equal to the total number of solutions of the equation $u_1 + u_2 + \ldots + u_r = n$, which is $\binom{n-1}{r-1}$. Therefore, for binary alphabet, $|\cS(n,r)| = 2|\bbS(n,r)| = 2\binom{n-1}{r-1}$. In general, the constrained space has size $q(q-1)^{r-1}\binom{n-1}{r-1}$ for $q$-ary alphabet. We remark that our analysis only deals with runs, hence it is still correct for $q$-ary alphabet. The result for the general case will be discussed in the extended version.  

The capacity for the binary case is in the below proposition.

\begin{proposition}\label{prop:run-cap}
For fixed $0\le \rho\le 1$, we have that $$\bCap(\rho)= \lim_{n\to\infty}\frac{\log |\cS(n,\floor{\rho n})|}{n}=\HH(\rho).$$
\end{proposition}

Next, we formally define a sticky-insertion-correcting code. For convenience, we introduce the notion of confusability.
\begin{definition}\label{def:confusable}
	Two words $\bu,\bv \in \bbS(n,r)$ are $b$-confusable if there exists $\bw \in \bbS(n+b,r)$ such that $\bw$ can be obtained from both $\bu$ and $\bv$ via increasing their coordinates by $b$ units. 
\end{definition}
For example, $\bu = (2,3), \bv=(1,4) \in \bbS(5,2)$ are $1$-confusable since $\bw = (2,4)$ can be obtained from $\bu$ by adding one to its second coordinate or from $\bv$ by adding one to its first coordinate. 

We say that $C\subseteq \bbS(n,r)$ is an {\em $(n,r,b)$-sticky-insertion code} if $\bu$ and $\bv$ are not $b$-confusable for any pair of distinct codewords $\bu,\bv\in C$.
Suppose we have codes $C_1,C_2,\ldots, C_{n}$ such that $C_r$ is an $(n,r,b)$-sticky-insertion code for $1\le r\le n$.
Then any pair of distinct codewords in  $C = \bigcup_{r=1}^{n} C_r$ are not $b$-confusable too. 


Let $A_{SI}(n,r,b)$ be the size of a largest $(n,r,b)$-sticky-insertion code and 
we set $A_{SI}(n,b)=\sum_{r=1}^n A_{SI}(n,r,b)$.
Bounds on $A_{SI}(n,b)$ were first studied in \cite{Levenshtein1965}%
\footnote{Unlike Sections~\ref{sec:GV} and~\ref{sec:Synthesis}, $A_{SI}$ is a function of the number of correctable errors, and not a function of the minimum distance. This is to be consistent with the notation of previous work.}.
A different construction of codes without the constraint of runs was subsequently given in \cite{dolecek2010repetition} and \cite{mahdavifar2017}. 
Recently, in \cite{kovavcevic2018}, the authors 
obtained the following upper and lower bounds on $A_{SI}(n,b)$. 
\begin{equation*}
	\frac{2^{n+b}}{n^b} \lesssim A_{SI}(n,b) \lesssim \frac{2^{n+b+s}}{n^b}s!(b-s)!
\end{equation*}
Here, $s = \floor{\frac{b+1}{3}}$.
We also restrict to the space with runs constraint $r=\floor{\rho n}$ for fixed $\rho$ as in \cite{kovavcevic2018}. 
In contrast, we allow $b=\floor{\beta n}$ to grow with $n$. The difference is that $\lim_{n\to \infty}\frac{\log{|A_{SI}(n,b)|}}{n} = 1$, whenever $b$ is a constant.
This section aims to obtain the GV lower bound for $$ \alpha_{SI}(\beta)\triangleq \lim_{n\to \infty}\frac{\log{|A_{SI}(n,\beta n)|}}{n}.$$ 
We recall that the $L_1$-{\em distance} between $\bu = (u_1,u_2,\ldots ,u_r)$ and $\bv = (v_1,v_2,\ldots, v_r)$ is $D(\bu,\bv) \triangleq \sum_{i=1}^{r}|u_i-v_i|$. The $L_1$-{\em distance} fully characterizes $b$-confusability for this channel.


\begin{lemma}\label{Lem:d-conf}
	$\bu, \bv \in \bbS(n,r)$ are $b$-confusable if and only if $D(\bu,\bv) \leq 2b$.
\end{lemma}
\begin{proof}
Suppose that $\bu$ and $\bv$ are $b$-confusable. Then there exists $\bw$ that can be obtained by introducing $b$ sticky-insertions from both $\bu$ and $\bv$.
Hence, $D(\bu,\bw) = b$ and $D(\bv,\bw) = b$.
Therefore, $D(\bu,\bv) \leq D(\bu,\bw) + D(\bv,\bw) = 2b$, as required.

Conversely, suppose that $D(\bu,\bv) = 2b' \leq 2b$.
Since $\sum_{i} u_i = \sum_{i} v_i=n$, we have that
\[\sum_{u_i > v_i} (u_i - v_i) = \sum_{v_i > u_i} (v_i - u_i) = b'.\] 

We construct $\bw = (w_1,w_2,\ldots,w_r)$ such that $w_i =  {\max}(u_i,v_i)$. Then,
\[
	D(\bu,\bw) = \sum_{i=1}^{r}|w_i - u_i| = \sum_{v_i > u_i}| v_i - u_i| = b' \leq b. 
\]
\[	D(\bv,\bw) = \sum_{i=1}^{r}|w_i - v_i| = \sum_{u_i > v_i}| u_i - v_i| = b' \leq b. 
\]
Therefore, $\bw$ can be obtained from $\bu$ and $\bv$ via $b$ sticky-insertions.
And hence, they are $b$-confusable by Definition~\ref{def:confusable}.
\end{proof}

\subsection{Total Ball Size}

{
	We consider balls with center $\bu \in \bbS(n,r)$ and radius $2b$, that is, 
	$V(\bu,2b) = \{\bv \in \bbS(n,r): D(\bu,\bv) \leq 2b\}$. 
	Then Lemma \ref{Lem:d-conf} states that $\bu$ and $\bv$ are $b$-confusable if and only if $v \in V(\bu,2b)$.
	
	Since the space $\bbS(n,r)$ is specified by word length $n$ and the number of runs $r$, 
	we consider the total ball $T(n,r,d) = \{(\bu,\bv) \in \bbS(n,r)^2: D(\bu,\bv) \leq d\}$ and our task is to determine $\widetilde{T}(\rho,\delta) \triangleq \limsup_{n \rightarrow \infty} \frac{\log |T(n,\floor{\rho n},\floor{\delta n})|}{n}$ where $\delta=2\beta$. 
	
	To this end, we consider the number of pairs $(\bu,\bv)$ of $L_1$ distance exactly $s$, denoted by $N(n_1,n_2,r,s) = |\{(\bu,\bv) \in \bbS(n_1,r)\times \bbS(n_2,r): D(\bu,\bv) = s\}|$. Here, we propose the following lemma to recursively count $N(n_1,n_2,r,s)$. We note that $N(n_1,n_2,r,s) = 0$ if one of $n_1,n_2,r,s$ is negative. 
	\begin{lemma}\label{lem:recur_stick}
		\begin{align*}
			N(n_1,n_2,r,s) &= \sum_{i \geq 1}N(n_1-i,n_2-i,r-1,s) \\ 
			&\hspace{4mm} + \sum_{i \geq 1}\sum_{j \geq 1}N(n_1-i,n_2-i-j,r-1,s-j) \\
			&\hspace{4mm} + \sum_{i \geq 1}\sum_{j \geq 1}N(n_1-i-j,n_2-i,r-1,s-j)\,.  
		\end{align*}
	\end{lemma}
	\begin{proof}
		Let $\bu= (u_1,\ldots, u_r)$ and $\bv=(v_1,\ldots,v_r) \in N(n_1,n_2,r,s)$. We consider truncating the last run $u_r$ and $v_r$. If $u_r=v_r=i$ for $i\geq 1$, we get the first sum where the distance remains the same. Otherwise, $u_r = i$ and $v_r = i+j$ for $i,j \geq 1$. Here, the length of $\bu,\bv$ become $n_1-i$ and $n_2-i-j$ respectively. Their distance decreases by $|u_r-v_r|=j$. Hence, we get the second term. The last one is obtained similarly when $u_r > v_r$.
	\end{proof}
	With this recursion, we are ready to find the generating function $F(x_1,x_2,y,z) \triangleq \sum_{n_1,n_2,r,s \geq 0} N(n_1,n_2,r,s)x_{1}^{n_1}x_{2}^{n_2}y^rz^s$.
\begin{lemma}\label{lem:generateN}
$F(x_1,x_2,y,z) = \frac{G(x_1,x_2,y,z)}{H(x_1,x_2,y,z)}$,
where 
{
\begin{align*}
	G &= (1-x_1x_2)(1-x_1z)(1-x_2z)\,, \\ 
	H &= (1-x_1x_2)(1-x_1z)(1-x_2z) - yx_1x_2(1-x_1x_2z^2)\,.  
\end{align*}
}
\end{lemma}
	
	\begin{proof}
		{\footnotesize
			\begin{align*}
				&F(x_1,x_2,y,z)\\
				&= \sum_{n_1,n_2,r,s \geq 0}N(n_1,n_2,r,s)x_{1}^{n_{1}}x_{2}^{n_{2}}y^rz^s \\ 
				&= \sum_{n_1,n_2,r,s \geq 0}\sum_{i \geq 1}N(n_1-i,n_2-i,r-1,s)x_{1}^{n_1}x_{2}^{n_2}y^rz^s \\ 
				&\hspace{2mm} + \sum_{n_1,n_2,r,s \geq 0}\sum_{i \geq 1}\sum_{j \geq 1}N(n_1-i,n_2-i-j,r-1,s-j)x_{1}^{n_1}x_{2}^{n_2}y^rz^s \\ 
				&\hspace{2mm} + \sum_{n_1,n_2,r,s \geq 0}\sum_{i \geq 1}\sum_{j \geq 1}N(n_1-i-j,n_2-i,r-1,s-j)x_{1}^{n_1}x_{2}^{n_2}y^rz^s
				\\ &=  1 + F(x_1,x_2,y,z)(y\sum_{i \geq 1}(x_1x_2)^i)(1 + \sum_{j \geq 1}(x_1z)^j + \sum_{j \geq 1}(x_2z)^j)\,.
			\end{align*}
		}
		Hence,
		{\footnotesize
		\begin{align*}
		\hspace{-1mm} F(x_1,x_2,y,z)
			&= \frac{1}{1-(y\sum_{i \geq 1}(x_1x_2)^i)(1 + \sum_{j \geq 1}(x_1z)^j + \sum_{j \geq 1}(x_2z)^j)} \\ &= \frac{(1-x_1x_2)(1-x_1z)(1-x_2z)}{(1-x_1x_2)(1-x_1z)(1-x_2z) - yx_1x_2(1-x_1x_2z^2)}\,.
		\end{align*}
		}
	\end{proof}

	From Theorem~\ref{thm:acsv}, with $n_1=n_2=n, r= \rho n, s = \delta n$, we solve the following system of equations.  We denote the partial derivates $\frac{\partial H}{\partial x}$ as $H_x$.
	\begin{equation}\label{eq:stickH}
			H = 0 \text{\quad and\quad}
			x_{1}H_{x_1} = x_{2}H_{x_2} = \frac{yH_{y}}{\rho} = \frac{zH_{z}}{\delta} 
	\end{equation}
	
	\begin{lemma}\label{lem:solveH_sticky}
		The solution of the equation system (\ref{eq:stickH}) is
		{\footnotesize
		\begin{align*}
			x^*(\delta) &= x_{1}^{*}(\delta) = x_2^{*}(\delta) =  \sqrt{1 - \frac{2\rho}{2-\delta}},\\
			z^*(\delta) &= \frac{\sqrt{\rho^2 + \delta^2} - \rho}{x^*\delta},\\
			y^*(\delta) &= 2\frac{\sqrt{\rho^2 + \delta^2} - \delta}{2-\delta - 2\rho}  
		\end{align*}
		}
	\end{lemma}
	Applying \eqref{eq:asym_approx}, we have that
	\begin{align*}
	\lim_{n \to \infty} \frac{\log N(n,n,\rho n,\delta n)}{n} &= -2\log x^*(\delta) - \rho \log y^*(\delta) \\ & \hspace{3mm} - \delta \log z^*(\delta)\,.  
	\end{align*}
	
		From Lemma \ref{Lem:d-conf}, with $\beta n$ correctable errors, we need to consider the total ball size with distance $\delta n$ where $\delta=2\beta$. This quantity is $|T(n,\rho n,\delta n)| = \sum_{s=0}^{\delta n}N(n,n,\rho n, s)$. Hence, we have that 
	{\small 
	\begin{align*}
	&\widetilde{T}(\rho,\delta) \\
	& = \max_{0\le \beta_1 \le \beta} -2\log x^*(2\beta_1) - \rho \log y^*(2\beta_1) - 2\beta_1 \log z^*(2\beta_1) \\
	& = \begin{cases}
	-2\log x^*(2\beta) - \rho \log y^*(2\beta) - 2\beta \log z^*(2\beta)\,, &\hspace{-3mm} \text{if }\beta\le \beta_{\rm{max}}\,,\\
	2\HH(\rho)\,, & \hspace{-3mm}\text{if }\beta\ge \beta_{\rm{max}}\,.
	\end{cases}
	\end{align*}
	}
	Here, $\beta_{\rm{max}}=(1-\rho)/(2-\rho)$. 
	
	In conclusion, we have the following explicit formula for the asymptotic ball size.
	
\begin{corollary}\label{cor:rate_sticky}
For fixed $\rho$, set $\beta_{\rm{max}}=(1-\rho)/(2-\rho)$.
When $\beta\leq\beta_{\rm{max}}$, we have
\begin{align*}
\widetilde{T}(\rho,2\beta) &= -\rho + 2\beta \log(2\beta) - \rho \log(\sqrt{\rho^2+4\beta^2} - 2\beta) \\
					&\hspace{4mm}- 2\beta \log(\sqrt{\rho^2+4\beta^2}-\rho) \\
					&\hspace{4mm} +(-1+\rho+\beta) \log(2-2\rho-2\beta) \\
					&\hspace{4mm}+ (1-\beta)\log(2-2\beta)
\end{align*}

Otherwise, when  $\beta>\beta_{\rm{max}}$, we have $\widetilde{T}(\rho,\beta)=2\HH(\rho)$.
\end{corollary}

	
Hence, $\lim_{n\to\infty} \frac{\log A_{SI}(n,\rho n,\beta n)}{n} \geq 2 \HH(\rho) - \widetilde{T}(\rho,2\beta)$.
Since $ A_{SI}(n,\beta n)=\sum_{r=1}^n A_{SI}(n, r, \beta n)$, 
we optimize the right-hand side over $0\le \rho\le 1$.
We have the following result.
	
\begin{proposition}\label{prop:gv-si}
For fixed $\beta > 0$, we have $\alpha_{SI}(\beta)\geq R^{(SI)}_{\rm{GV}}(\beta)$, where
$R^{(SI)}_{\rm{GV}}(\beta) \triangleq 2\HH(\rho) - \widetilde{T}(\rho,2\beta)$ and 
$\rho = \frac{3(1-\beta) - \sqrt{9\beta^2-2\beta+1}}{4}$.
\end{proposition}

\begin{remark}
	Proposition~\ref{prop:gv-si} states that for all $\beta<1/2$, there exists a family of $(n,\beta n)$-sticky-insertion codes with a positive rate. 
	Since a code that corrects $b$ sticky-insertions also corrects $b$ run-preserving deletions (see for example, \cite{kovavcevic2018}), we also have a family of $(n,\beta n)$-run-preserving-deletion codes with positive rates whenever $\beta<1/2$. 
	We emphasize the situation is different for general deletion-correcting codes. Recently, the authors in~\cite{Guruswami2022.deletionthreshold} showed that there exists $\beta^*<1/2$ such the rate of any $(n,\beta^* n)$-deletion-correcting code must be zero.
\end{remark}

\subsection{Numerical Plots}

\begin{figure}[t!]
	\begin{center}
		\includegraphics[width=0.8\linewidth]{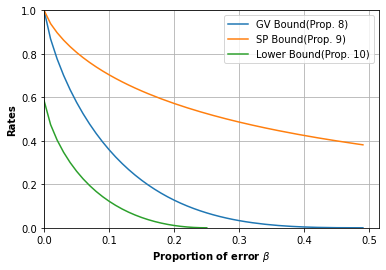}
	\end{center}
	\caption{Bounds for $\alpha_{SI}(\beta)$ for the sticky-insertion channel.}
	\vspace{-2mm}
	\label{fig:sticky}
\end{figure}

In this section, we compare the bound in Proposition~\ref{prop:gv-si} with a sphere-packing bound and a simpler lower bound.

\vspace{1mm}

\noindent{\em Sphere-Packing Bound}. Given $\bu\in \bbS(n,r)$, the resulting output $\bv$ with $b$ sticky insertions belongs to $\bbS(n+b,r)$. 
Furthermore, $\bv$ belongs to the set 
$\{\bv' \in \bbS(n+b,r): v_i' \ge u_i \text{ for all } i \}$
of size $\binom{r+b-1}{r-1}$. 
Therefore, the sphere-packing bound is $$A_{SI}(n,b) \leq \sum_{r=1}^{n-1} 2\frac{\binom{n+b-1}{r-1}}{\binom{r+b-1}{r-1}}.$$
Asymptotically, we have the following proposition.

\begin{proposition}\label{prop:sp-si}
	For fixed $\beta > 0$, we have that $\alpha_{SI}(\beta)\leq R^{(SI)}_{\rm{SP}}(\beta)$, where
	$R^{(SI)}_{\rm{SP}}(\beta) \triangleq (1+2\beta)(1-\HH(\frac{1+\beta}{1+2\beta}))$.
\end{proposition}

\vspace{2mm}

\noindent{\em Simpler Lower Bound}.
We describe a crude upper bound on the total ball size $|T(n,r,d)|$, leading to a simpler lower bound on $\alpha_{SI}(\delta)$. To this end, we recall that $T(n,r,d) = \{(\bu,\bv) \in \bbS(n,r)^2: D(\bu,\bv) \leq d\}$ and this is a subset of $\{(\bu,\bv) \in \bbS(n,r)\times \cup_{i=-b}^{b}\cS(n+i,r): D(\bu,\bv) \leq d\}$. This new quantity can be obtained by distributing $d$ into $r$ parts as $(d_1,d_2,\ldots,d_r)$ and then assigning $+$ or $-$ to each coordinate. Hence, $|T(n,r,d)| \le 2^r\binom{d+r-1}{r-1}$.
As before, we set $\delta=2\beta$ and we have a weaker lower bound.


\begin{proposition}\label{prop:lb-si}
For fixed $\beta > 0$, we have $\alpha_{SI}(\beta)\geq 2\beta-1 - (1+2\beta)\log(\frac{1+2\beta}{3}) + 2\beta\log \beta$.
\end{proposition}
In Figure~\ref{fig:sticky}, the GV bound obtained from the sharp estimate of $|T(n,r,d)|$ is significantly larger than
the bound obtained by the simple lower bound of $|T(n,r,d)|$.

\section{GV Bound for the Synthesis Channel}\label{sec:Synthesis}

{
\begin{figure}[t!]
	\begin{center}
		\includegraphics[width=0.75\linewidth]{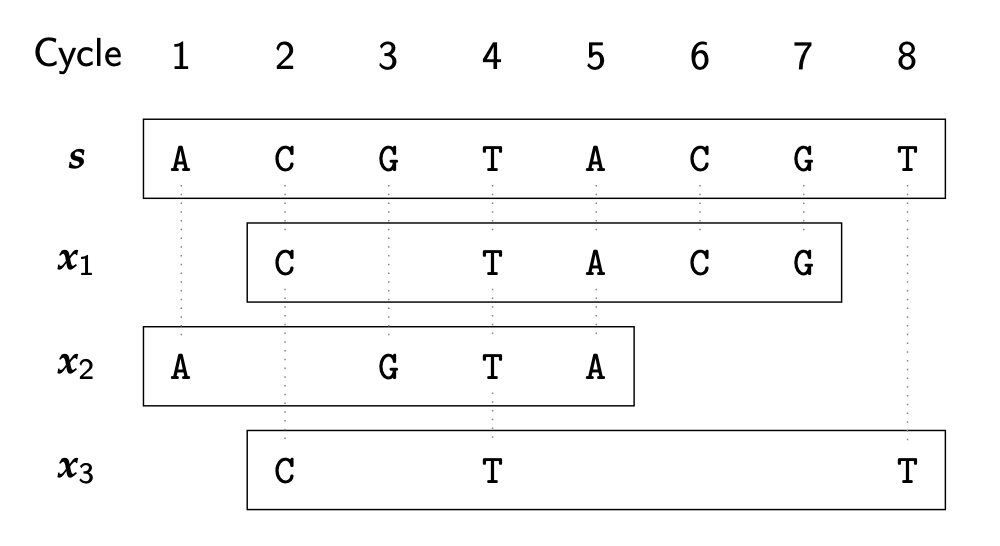}
	\end{center}
	\caption{Synthesis of three strands $\bx_1=\C\T\A\C\G$, $\bx_2=\A\G\T\A$, and $\bx_3=\C\T\T$ using the supersequence $\bs=(\A\C\G\T~\A\C\G\T)$. The strand $\bx_1$ is synthesized in cycles $2,4,5,6,7$, the strand $\bx_2$ is synthesized in cycles $1,3,4,5$ and $\bx_3$ is in cycles $2,4,8$.}
	\label{fig:syn_process}
\end{figure}
This section describes the procedure to compute the GV bound for the {\em DNA synthesis} channel. In this section, $\Sigma = \{\A,\C,\G,\T\}$.
Recently, DNA has emerged as a next-generation data storage medium because of its unprecedented density, durability, and replication efficiency \cite{yazdi2015dna}. This work considers the synthesis process, which is essential in embedding data into DNA. 
In particular, digital data is preprocessed and encoded in physical DNA molecules using synthesis machines. Iterating through a fixed supersequence $\bs = (s_1,s_2, \ldots ) \in \Sigma^*$ of nucleotides, 
the machine produces several DNA strands $\bx_1, \bx_2, \ldots \in \Sigma^*$ in parallel. These strands can be of equal or different lengths. 
In this paper, we focus on same-length strands $\bx_i\in \Sigma^n$. 
At each iteration/cycle, the machine either attaches $s_j$ to $\bx_i$ or not. 
Therefore, a DNA strand $\bx$ can be synthesized in $\cT$ cycles using the synthesis sequence $\bs$ if and only if $\bx$ is a subsequence of $(s_1,s_2, \ldots, s_{\cT})$. Figure~\ref{fig:syn_process} is an example of the synthesis process \cite{len2020coding}.
We consider sets of DNA strands so that the number of cycles needed to produce them is minimized. 


Formally, we consider the constrained space $\cS(n,\leq\!\cT)$ that comprises all length-$n$ subsequences of $(s_1,s_2, \ldots, s_{\cT})$. 
This coding problem was introduced by Lenz {\em et al}.~\cite{len2020coding} and {follow-up work include~\cite{Elishco2022, lenz2021multivariate, Makarychev2021}}.
In \cite{lenz2021multivariate}, the authors used multivariate combinatorics to determine the capacity of $\cS(n,\le \!\cT)$ and 
showed that the capacity is maximized when $\bs$ is an alternating sequence 
that cyclically repeats all symbols in $\Sigma$ in ascending order. 
Hence, this paper sets the supersequence $\bs$ to be the alternating sequence over the quaternary alphabet $(\A\C\G\T\A\C\G\T\ldots)$. For the fixed periodic supersequence, the synthesis time of a nucleotide is completely determined by its previous position. Hence, starting with an empty strand, the time to add a new $\A,\C,\G,\T$ to $\bx$ is $1,2,3,4$ respectively. For example, in Figure~\ref{fig:syn_process}, $\bx_1$ needs 7 cycles to be completely synthesized. It takes 2 cycles to get the first $\C$, 2 more cycles to go from $\C$ to $\T$, and so on. The number of cycles can be computed iteratively as $2+(4-2)+1+(2-1)+(3-2)$. We note that $\bx_1\in \cS(5,7)$, $\bx_2\in \cS(4,5)$, and $\bx_3\in \cS(3,8)$.}

The following proposition obtained by Lenz~{\em et al.} will be useful for the GV bound.
\begin{proposition}[{\cite[Proposition~6.7]{lenz2021multivariate}}]\label{prop:synthesis-cap}
	Fix $\tau$ and define $\bCap(\tau) \triangleq \lim_{n\to\infty} \frac{\log |\cS(n,\le \floor{\tau n})|}{n} $. Then
\begin{align*}
	\bCap(\tau) 
	& = \max_{0\le \tau_1 \le \tau} -\log \bar{x} - \tau_1 \log \bar{y} \\
	& = \begin{cases}
		-\log \bar{x} - \tau \log \bar{y} \,, & \text{ if }\tau < 5/2\,,\\
		2 \,, & \text{ otherwise }.
	\end{cases}
\end{align*}
Here, $\bar{x} = \frac{1}{\bar{y}+\bar{y}^2+\bar{y}^3+\bar{y}^4}$ and
$\bar{y}$ is the unique real root of polynomial $(4-\tau)y^3 + (3-\tau)y^2 + (2-\tau)y + (1-\tau) = 0$.
\end{proposition}

In this work, we introduce error-correcting capabilities to this constrained space.
In particular, we study codes that correct {\em substitution errors} and whose words have bounded synthesis time.
Formally, we say that $C \subseteq \cS(n, \le\!\cT)$ is a $(n,\le\!\cT,d)$-{\em synthesis code} if any pair of distinct $\bu,\bv \in C$ have Hamming distance at least $d$. 
As before, we want to determine $A_{SY}(n,\le\!\cT,d)$, the size of a largest $(n,\le\!\cT,d)$-synthesis code, 
and its asymptotic rate $\alpha_{SY}(\tau,\delta) = \lim_{n\to\infty} \frac{\log |A_{SY}(n,\le \floor{\tau n},\floor{\delta n})|}{n}$.

\subsection{Total Ball Size}


Specifically, we consider the set $T(n,d,\cT) = \{(\bu,\bv) \in \bbS(n,\le\!\cT)^2: D_H(\bu,\bv) \leq d\}$ and our task is to determine $\widetilde{T}(\tau,\delta) \triangleq \limsup_{n\rightarrow \infty} \frac{\log |T(n,\floor{\tau n},\floor{\delta n})|}{n}$.

To this end, we consider the quantity 
{
\footnotesize
\begin{align*}
	N(n,t,s) &\triangleq \left|\{(\bu,\bv) \in \bigcup_{t_1=0}^t \cS(n,=t_1) \times \cS(n,=t-t_1):D_H(\bu,\bv) = s\}\right|\,,
\end{align*}
}
\noindent Here, $\cS(n,t)$ denote the set of all length-$n$ quaternary sequences with synthesis time exactly $t$. Thus, $|T(n,d,\cT)|$ is upper bounded by the sum $\sum_{t=0}^{2\cT}\sum_{s=0}^d N(n,t,s)$.
Next, we have the following lemma that recursively computes $N(n,t,s)$.

\begin{lemma}\label{lem:generateN_synthesis}
{\small
	\begin{align*}
		N(n,t,s) &= \sum_{i = 1}^{4}N(n-1,t/2-i+ t/2-i,s) \\ 
		&\hspace{1mm} + 2\sum_{i = 1}^{3}\sum_{j = i+1}^{4}N(n-1,t/2-i+t/2-j,s-1) \,.
	\end{align*}
}
\end{lemma}
\begin{proof}
We consider the first synthesized nucleotide $u_1,v_1$ of $\bu,\bv \in N(n,t_1+t_2,s)$ respectively. When $u_1=v_1$, the distance between $\bu,\bv$ remains $s$. If we remove them, the length is $n-1$ and processing time reduces by $i \in \{1,2,3,4\}$. When $u_1\neq v_1$, the distance becomes $s-1$ and processing time decreases by $i+j$. The order $(i,j) \neq (j,i)$, so we get the factor 2. 
\end{proof}
As before, we can determine the corresponding generating function  
$F(x,y,z) = \sum_{n,t,s \geq 0}N(n,t,s)x^{n}y^{t}z^{s}$

\begin{lemma}\label{lem:generateN-syn}
	$F(x,y,z) = \frac{1}{H(x,y,z)}$,
	where 
	\[
			H = 1-xy^2(1+y^2)((1+y^4)+2zy(1+y+y^2))\,.  
	\]	
\end{lemma}
\vspace{-3mm}
As before, we solve the following system of equations.
\begin{equation}\label{eq:synthesisH}
			H = 0 \text{ and }
		xH_{x} = \frac{yH_{y}}{\tau} = \frac{zH_{z}}{\delta}. 
\end{equation}

\begin{lemma}\label{lem:solveH_synthesis}
	The solution of the equation system (\ref{eq:synthesisH}) is
	\begin{align*}
		\hat{x} =  \frac{1-\delta}{\hat{y}^2(1+\hat{y}^2)(1+\hat{y}^4)}\, , 
		\hat{z} = \frac{\delta(1+\hat{y}^4)}{2(1-\delta)\hat{y}(1+\hat{y}+\hat{y}^2)}\,, 
	\end{align*}
where $\hat{y}$ is the smallest positive real solution of the equation
{\small
\begin{align*}
	&\tau (1+y^2)(1+y^4)(1+y+y^2) \\ 
	&=2(1+y+y^2)(1+2y^2+3y^4+4y^6) + \delta (1-y^4)(1+y^2+y^4).
\end{align*}
}
\end{lemma}

Applying \eqref{eq:asym_approx}, we have that
\begin{equation*}
	\lim_{n \to \infty} \frac{\log N(n,\tau n,\delta n)}{n} = -\log \hat{x} - \tau \log \hat{y} - \delta \log \hat{z}\,.  
\end{equation*}

Recall that $|T(n,\tau n,\delta n)| \le  \sum_{t=0}^{2\tau n}\sum_{s=0}^{\delta n}N(n, t, s)$. Hence,
\begin{align*}
	&\widetilde{T}(\tau,\delta) \\
	& \le \max_{0\le \tau_1 \le 2\tau}\max_{0\le \delta_1 \le \delta} -\log \hat{x} - \tau_1 \log \hat{y} - \delta_1 \log \hat{z} \\
	& = \begin{cases}
		-\log \hat{x} - 2\tau \log \hat{y} - \delta \log \hat{z}\,, & \text{ if }\tau < 5/2, \delta\le \delta_{\rm{max}}\,,\\
		2\bCap(\tau) \,, & \text{ if }\tau < 5/2, \delta\geq \delta_{\rm{max}}\,,\\
		2 + \HH(\delta) + \delta \log 3 \,, & \text{ if }\tau \geq 5/2, \delta\le 3/4\,,\\
		4 \, & \text{ if }\tau \geq 5/2, \delta\geq 0.75\,.
	\end{cases}
\end{align*}
Here, 
\begin{align*}
&\delta_{\rm{max}} = \frac{2y_{\rm{min}}(1+y_{\rm{min}}+y_{\rm{min}}^{2})}{(1+y_{\rm{min}}^{4})+2y_{\rm{min}}(1+y_{\rm{min}}+y_{\rm{min}}^{2})},
\end{align*}
and $y_{\rm{min}}$ is the smallest positive real solution of the equation
{\footnotesize
\begin{align*}
\hspace*{-2mm} \frac{y}{(1+y^4)+2y(1+y+y^2)} = \frac{\tau(1+y^2)(1+y^4)-(4y^6+3y^4+2y^2+1)}{(1-y^4)(y^4+2y^3+4y^2+2y+1)}.
\end{align*}
}

Finally, we obtain the following lower bound for $\alpha_{SY}(\tau,\delta)$.
\begin{proposition}\label{prop:gv-sy}
	For fixed $\tau,\delta > 0$, consider the above upper bound for $\widetilde{T}(\tau,\delta)$.
	Then we have that $\alpha_{SY}(\tau,\delta)\geq R^{(SY)}_{\rm{GV}}(\tau,\delta)$, where
	$R^{(SY)}_{\rm{GV}}(\tau,\delta) \triangleq 2\bCap(\tau) - \widetilde{T}(\tau,\delta)$.
\end{proposition}

\subsection{Numerical Plots}

\begin{figure}[!t]
\hspace{-4mm} \begin{minipage}{0.5\linewidth}
		\centering
		\includegraphics[scale=0.35]{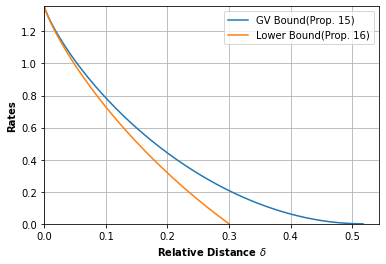}
		\caption*{$\tau = 1.5$}
		
	\end{minipage}%
	\begin{minipage}{0.5\linewidth}
		\centering
		\includegraphics[scale=0.35]{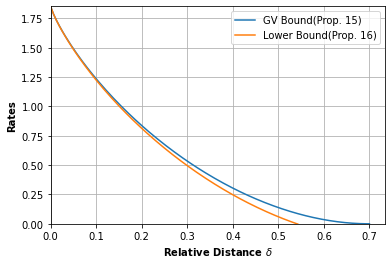}
		\caption*{$\tau=2$}
	\end{minipage}
	\caption{Bounds for  $\alpha_{SY}(\tau,\delta)$ for the {synthesis} channel.}
	\label{fig:synthesis}\end{figure}

In this section, we plot the GV bound for the {\em synthesis channel}. For comparison purposes, we also obtain the plot of a simpler lower bound. 
Specifically, we have that the following crude upper bound, $|T(n,d,\cT)| \le \binom{n}{d} 3^{d}$.
Hence, we obtain the asymptotically lower bound for $\alpha_{SY}(\tau,\delta)$.
 \begin{proposition}\label{prop:lb-sy}
 	For fixed $\delta > 0$, we have that $\alpha_{SY}(\tau,\delta)\geq R^{(SY)}_{\rm{LB}}(n,\delta)$, where
 	$R^{(SY)}_{\rm{LB}}(n,\delta) \triangleq \bCap(\tau) - \HH(\delta) - \delta\log 3$\,.
 \end{proposition}
Figure~\ref{fig:synthesis} illustrates that the improved estimate of $|T(n,r,d)|$ provides a better GV bound compared to the bound obtained by the simple bound of $|T(n,r,d)|$.
}

\section{Acknowledgement}
The work of Han Mao Kiah was supported by the Ministry of Education, Singapore, under its MOE AcRF Tier 2 Award MOE-T2EP20121-0007.

\newpage


\newpage
\appendices
\section*{Appendix}
{
\section{Proof of Lemma \ref{lem:solveH_sticky}}\label{appen_solve_sticky}
\vspace*{-2mm}
We need to find the positive solution of
\vspace*{-1mm}
\begin{equation}\label{eq:stickyH}
\vspace*{-2mm}
\begin{split}
H &= 0 \\
 x_{1}H_{x_1} = x_{2}H_{x_2}& = \frac{yH_{y}}{\rho} = \frac{zH_{z}}{\delta} 
 \end{split}
\end{equation}
where $H = (1-x_1x_2)(1-x_1z)(1-x_2z) - yx_1x_2(1-x_1x_2z^2)$.
\vspace*{-2mm}
\begin{proof}
Take the partial derivatives as
\vspace*{-1mm}
{\footnotesize
\begin{align*}
H &= (1-x_1x_2)(1-x_1z)(1-x_2z) - yx_1x_2(1-x_1x_2z^2),  \tag{i} \\
x_1H_{x_1} &= -x_1(1-x_2z)(x_2 + z - 2x_1x_2z) - yx_1x_2(1-2x_1x_2z^2), \tag{ii}\\
x_2H_{x_2} &= -x_2(1-x_1z)(x_1 + z - 2x_1x_2z) - yx_1x_2(1-2x_1x_2z^2), \tag{iii}\\
\frac{zH_z}{\delta} &= \frac{-z(1-x_1x_2)(x_1 + x_2 - 2x_1x_2z) - yx_1x_2(-2x_1x_2z^2)}{\delta}, \tag{iv}\\ 
\frac{yH_y}{\rho} &= \frac{-yx_1x_2(1-x_1x_2z^2)}{\rho}  \tag{v}.
\end{align*}
}
Firstly, equating (i) and (ii) gives $z(x_1-x_2)(1-x_1x_2) = 0$. As $z=0$ is not positive. If $x_1x_2=1$, then $0= H = y(1-z^2)$, leading to $z=1$. Combining with (iv) and (v), we obtain $2y/\delta = 0$ i.e. $y=0$, which is not feasible. Therefore, $x_1 = x_2$.

Secondly, since $H=0$, we substitute $y$ in (iv) and (v) by
\begin{equation}\label{eq:Ystar}
y = \frac{(1-x_1^2)(1-x_1z)^2}{x_1^2(1-x_1^2z^2)} = \frac{(1-x_1^2)(1-x_1z)}{x_1^2(1+x_1z)}.
\end{equation}

Simplify to
{\small
\begin{align*}
zH_z &= -z(1-x_1^2)(2x_1-2x_1^2z) +  2\frac{(1-x_1^2)(1-x_1z)}{x_1^2(1+x_1z)} x_1^4z^2 \\
	&= -2zx_1(1-x_1^2)(1-x_1z) +2 \frac{(1-x_1^2)(1-x_1z)}{(1+x_1z)} x_1^2z^2 \\
	&= -\frac{2zx_1}{1+x_1z} \Big[ (1-x_1^2)(1-x_1^2z^2) - x_1z(1-x_1^2)(1-x_1z) \Big] \\
	&= -\frac{2zx_1(1-x_1^2)(1-x_1z)}{1+x_1z}, \\
yH_y &= -\frac{(1-x_1^2)(1-x_1z)}{x_1^2(1+x_1z)} x_1^2(1-x_1^2z^2)\\
	&= -\frac{(1-x_1^2)(1-x_1z)(1-x_1^2z^2)}{1+x_1z}. 
\end{align*}
}
By equating (iv) and (v), we get 
\begin{align*}
\frac{2zx_1(1-x_1^2)(1-x_1z)}{\delta} &=  \frac{(1-x_1^2)(1-x_1z)(1-x_1^2z^2)}{\rho} \\
\text{Equivalent to }
\frac{2zx_1}{\delta} &=  \frac{1-x_1^2z^2}{\rho} \text{as in (\ref{eq:Ystar}) $y > 0$.} 
\end{align*}
To ensure $x_1$ and $z$ be positive,
\begin{equation}\label{eq:Zstar}
x_1z 
= \frac{-\rho + \sqrt{\rho^2 + \delta^2}}{\delta}
\end{equation}
Let's substitute $y$ in (ii) by using the equation (\ref{eq:Ystar}).
{\footnotesize
\begin{align*}
\hspace*{-4mm} x_1H_{x_1} &= -x_1(1-x_1z)(x_1+z-2x_1^2 z) -\frac{(1-x_1^2)(1-x_1z)}{x_1^2(1+x_1z)}x_1^2(1-2x_1^2z^2)\\
&= -\frac{1-x_1z}{1+x_1z}\Big[ (x_1^2+x_1z-2x_1^3z)(1+x_1z)+ (1-x_1^2)(1-2x_1^2z^2)\Big]\\
&= -\frac{1-x_1z}{1+x_1z}\Big[1+x_1z -x_1^2z^2 -x_1^3z \Big]
\end{align*}
}
\vspace*{-1mm}
Lastly, we equate (ii) and (v) to get 
\begin{align*}
1+x_1z -x_1^2z^2 - x_1^3z &= \frac{(1-x_1^2)(1-x_1^2z^2)}{\rho}, \\
\rho \Big[(1-x_1^2z^2) + x_1z(1-x_1^2)\Big] &= (1-x_1^2)(1-x_1^2z^2).\\
\text{Recall from (\ref{eq:Zstar}) that, } 1-x_1^2z^2 &= \frac{2\rho}{\delta} x_1z,\\
\text{So, } \rho\Big[\frac{2\rho}{\delta} x_1z + x_1z(1-x_1^2)\Big] &= \frac{2\rho}{\delta} x_1z(1-x_1^2),\\
\text{Equivalent to } \frac{2\rho}{\delta} + (1-x_1^2) &= \frac{2(1-x_1^2)}{\delta}  \text{as } x_1z > 0,\\
\text{Hence, } 1-x_1^2 &= \frac{2\rho}{2-\delta}.
\end{align*}
\vspace*{-2mm}
\begin{equation}\label{eq:Xstar}
\vspace*{-2mm}	
\text{In other words, } x_1 = \sqrt{1-  \frac{2\rho}{2-\delta}}.
\end{equation}
Therefore, Lemma \ref{lem:solveH_sticky} results from equations \ref{eq:Xstar}, \ref{eq:Zstar}, and \ref{eq:Ystar}.
\end{proof}
\vspace*{-2mm}
\section{Computation of Total Ball Size for Sticky Insertions}\label{appen:Tdelta}
When $\delta \leq \delta_{\max}$, $\widetilde{T}(\rho,\delta) = -2\log x - \rho \log y - \delta \log z$, where
\begin{equation*}
 x = \sqrt{1-  \frac{2\rho}{2-\delta}}, y = 2\frac{\sqrt{\rho^2+\delta^2}-\delta}{2-2\rho - \delta}, \text{and } z= \frac{\sqrt{\rho^2+\delta^2}-\rho}{x\delta}
\end{equation*}
We recall that
{\footnotesize
\begin{align*}
-2\log x &= -\log (2- 2\rho - \delta) + \log(2-\delta), \\
-\rho \log y &= -\rho - \rho \log(\sqrt{\rho^2+\delta^2} - \delta) +\rho\log(2-2\rho - \delta)),\\
-\delta \log z &= -\delta \Big( \log(\sqrt{\rho^2+\delta^2} - \rho) -\log x - \log \delta \Big) \\
			&= -\delta \log(\sqrt{\rho^2+\delta^2} - \rho) +\delta \log \delta + \delta \log x \\
			&= -\delta \log(\sqrt{\rho^2+\delta^2} - \rho) +\delta \log \delta + \\
			&\hspace{3mm}+\frac{\delta}{2} (\log (2-2\rho-\delta) - \log(2-\delta)), \\
\widetilde{T}(n,\rho,\delta) &= -\rho + \delta \log \delta - \rho \log(\sqrt{\rho^2+\delta^2} - \delta) - \delta \log(\sqrt{\rho^2+\delta^2}-\rho) \\
					&\hspace{3mm} +(-1+\rho+\delta/2) \log(2-2\rho-\delta) + (1-\delta/2)\log(2-\delta),\\
\widetilde{T}(n,\rho,2\beta) &= -\rho + 2\beta \log(2\beta) - \rho \log(\sqrt{\rho^2+4\beta^2} - 2\beta) \\
						&\hspace{3mm}- 2\beta \log(\sqrt{\rho^2+4\beta^2}-\rho) \\
					&\hspace{3mm} +(-1+\rho+\beta) \log(2-2\rho-2\beta) + (1-\beta)\log(2-2\beta). 
\end{align*}
}
When $\rho$ is fixed,
{\footnotesize
\begin{align*}
\frac{\partial \widetilde{T}(\rho,2\beta)}{\partial \beta} &= 2+2\log(2\beta) + \frac{2\rho}{\sqrt{\rho^2+4\beta^2}} \\	
											&\hspace{3mm} - 2\log(\sqrt{\rho^2+4\beta^2}-\rho) - \frac{8\beta^2}{\sqrt{\rho^2+4\beta^2}(\sqrt{\rho^2+4\beta^2}-\rho)} \\
											&\hspace{3mm} +\log(2-2\rho-2\beta) +1 -\log(2-2\beta)-1.
\end{align*}
}
Note that $\frac{2\rho}{\sqrt{\rho^2+4\beta^2}}-\frac{8\beta^2}{\sqrt{\rho^2+4\beta^2}\big(\sqrt{\rho^2+4\beta^2}-\rho\big)} = -2$.

Hence, $\frac{\partial \widetilde{T}(\rho,2\beta)}{\partial \beta} = 2\log(2\beta)- 2\log(\sqrt{\rho^2+4\beta^2}-\rho)+log(2-2\rho-2\beta) -\log(2-2\beta)= \log\Bigg(\frac{4\beta^2(2-2\rho-2\beta)}{(2-2\beta)\big(\sqrt{\rho^2+4\beta^2}-\rho\big)^2}\Bigg)$.

We equate $\frac{\partial \widetilde{T}(\rho,2\beta)}{\partial \beta} = 0$ to find $\beta_{\max}$,
\begin{align*}
&\text{Leading to } \log\Bigg(\frac{4\beta^2(2-2\rho-2\beta)}{(2-2\beta)\big(\sqrt{\rho^2+4\beta^2}-\rho\big)^2}\Bigg) = 0 \\
 &\text{Equivalent to } 4\beta^2(2-2\rho-2\beta)= (2-2\beta)\big(\sqrt{\rho^2+4\beta^2}-\rho\big)^2. 
\end{align*}

As $0<\rho, \beta<1$, the unique solution is  $\beta_{\max} = \frac{1-\rho}{2-\rho}$.


\section{Optimizations in Proposition~\ref{prop:gv-si}, \ref{prop:sp-si}, \ref{prop:lb-si}}
\subsection{Proposition~\ref{prop:gv-si}}
From Proposition~\ref{prop:gv-si}, $R^{(SI)}_{\rm{GV}}(\beta) \triangleq 2\HH(\rho) - \widetilde{T}(\rho,2\beta)$. Since
\begin{align*}
        \widetilde{T}(\rho,2\beta) 
				 = & 2\beta\log(2\beta) - \rho \log(\sqrt{\rho^2+4\beta^2} - 2\beta) \\
				&\hspace{3mm}- 2\beta \log(\sqrt{\rho^2+4\beta^2}-\rho) \\
				&\hspace{3mm}+(-1+\rho+\beta) \log(1-\rho-\beta) \\
				&\hspace{3mm}+ (1-\beta)\log(1-\beta).\\
\frac{\partial \widetilde{T}(\rho,\beta)}{\partial \rho} =& -\frac{\rho^2}{\sqrt{\rho^2+4\beta^2}\Big(\sqrt{\rho^2+4\beta^2}-2\beta\Big)}\\
										&\hspace{1mm} - \log \Big(\sqrt{\rho^2+4\beta^2}-2\beta\Big)\\
										& + \frac{2\beta}{\sqrt{\rho^2+4\beta^2}}+ \log(1-\rho-\beta) + 1.\\
\frac{\partial2\HH(\rho)}{\partial \rho} =& 2\log(1-\rho) - 2\log(\rho).\\
\text{\hspace{1mm}We obtain }\frac{\partial R^{(SI)}_{\rm{GV}}(\beta)}{\partial \rho} &= \log\Bigg(\frac{\big(\sqrt{\rho^2+4\beta^2}-2\beta\big)(1-\rho)^2}{\rho^2(1-\rho-\beta)}\Bigg).
\end{align*}

Equating $\frac{\partial R^{(SI)}_{\rm{GV}}(\beta)}{\partial \rho} = 0$, we get
\begin{align*}
\rho &= \frac{1-2\beta}{1-\beta} \text{, $R^{(SI)}_{\rm{GV}}(\beta)= 0$.}\\
\text{\hspace{3mm}Or } \rho &= \frac{3(1-\beta) + \sqrt{9\beta^2-2\beta+1}}{4}, \\ \text{ $R^{(SI)}_{\rm{GV}}(\beta)$ is maximized.}
\end{align*}
\subsection{Proposition~\ref{prop:sp-si}}

 Since the sphere-packing bound is $\vspace{2mm} \\ A_{SI}(n,b)  \le \sum_{r=1}^{n-1}2\frac{\binom{n+b-1}{r-1}}{\binom{r+b-1}{r-1}}$,
\begin{align}
	\alpha_{SI}(\beta) & \le \max_{0\le\rho\le 1}(1+\beta)\HH(\frac{\rho}{1+\beta}) - (\rho+\beta)\HH(\frac{\rho}{\rho+\beta}) \notag \\ 
	& 
	=  \max_{0\le\rho\le 1} \beta\log \beta + (1+\beta) \log (1+\beta) \notag \\ 
	& - (\rho + \beta) \log (\rho + \beta)  - (1+\beta - \rho)\log (1+\beta - \rho) \label{eq:sp} 
\end{align} 
 
For fixed $\beta$, the right hand side of equation \ref{eq:sp} {maximizes} at  
\begin{align*}
	1+\beta-\rho &= \rho + \beta \\ \implies \rho = \frac{1}{2}.
\end{align*}

Substituting $\rho = \frac{1}{2}$ back in equation~\ref{eq:sp}, we get
{\small
\begin{align*}
	\alpha_{SI}(\beta) & \le (1+2\beta) + \beta \log \beta + (1+\beta)\log(1+\beta) \\ & - (1+2\beta)\log(1+2\beta) \\ &= (1+2\beta)(1-\HH(\frac{1+\beta}{1+2\beta})) \\ & = R_{SP}^{SI}(\beta)
\end{align*}
}

\subsection{Proposition~\ref{prop:lb-si}}

Since the lower bound is $A_{SI}(n,b) \ge 2\frac{\binom{n-1}{r-1}}{2^{r}\binom{r+2b-1}{r-1}}$,
	\begin{equation}\label{eq:lb}
	\begin{split}
		\alpha_{SI}(\beta) & \ge \HH(\rho) - \rho - (\rho+2\beta)\HH(\frac{\rho}{\rho+2\beta}) \\
		& \notag= -(1-\rho) \log (1-\rho) - \rho - (\rho + 2\beta) \log (\rho + 2\beta) \\
		&\hspace{3mm}+ 2\beta \log (2\beta) 
	\end{split}
	\end{equation}
For fixed $\beta$, the right-hand side of the above equation maximizes at  
\begin{align*}
	\log \frac{1-\rho}{\rho + 2\beta} = 1.
\end{align*}
Therefore, we get $1 -\rho = 2(\rho + 2\beta) \implies \rho = \frac{1-4\beta}{3}$.
Substituting $\rho = \frac{1-4\beta}{3}$ back in equation~\ref{eq:lb}, we get
\begin{align*}
	\alpha_{SI}(\beta) & \ge 2\beta - 1 - (1+2\beta)\log(\frac{1+2\beta}{3}) + 2\beta\log \beta.
\end{align*}

\section{Proof of Proposition~\ref{prop:synthesis-cap}}

Applying Theorem \ref{thm:acsv}, we have $(\bar{x},\bar{y})$ is the root of the system of equations 
\begin{align*}
H_1 &= 0, \\
\tau x\frac{\partial H_1}{\partial x} &= y\frac{\partial H_1}{\partial y} . 
\end{align*}

where $H_1(x,y) = 1-x(y+y^2+y^3+y^4)$.
By solving these equations, we get the required solution.
\begin{align*}
H_1(x,y) &= 1 - x(y+y^2+y^3+y^4) = 0 & \\ \leftrightarrow x &= \frac{1}{y+y^2+y^3+y^4}.
\end{align*}
Besides,
\begin{align*}
\tau x\frac{\partial H_1}{\partial x}&=  - \tau x(y+y^2+y^3+y^4)  & \\ y\frac{\partial H_1}{\partial y} &= -xy(1+2y+3y^2+4y^3) 
\end{align*}
Equating them leads to 
\begin{align*}
&\hspace{6mm}\tau = \frac{1+2y+3y^2+4y^3}{1+y+y^2+y^3}.\\
 &\text{Equivalent to } (4-\tau)y^3 + (3-\tau)y^2 + (2-\tau)y + (1-\tau) = 0.
\end{align*}

\section{Proof of Lemma \ref{lem:generateN-syn}}\label{appen_generate_syn}
The generating function is given by
{\footnotesize
\begin{align*}
F(x,y,z)	&= \sum_{n,t,s \geq 0} N(n,t,s) x^n y^t z^s\\
	 	&= \sum_{n,t,s \geq 0}\Bigg(\sum_{i = 1}^{4}N(n-1,t/2-i+ t/2-i,s) \\
		&\hspace{2mm}+ 2\sum_{i = 1}^{3}\sum_{j = i+1}^{4}N(n-1,t/2-i+t/2-j,s-1) \Bigg) x^n y^t z^s\\
		&= F(x,y,z) \Bigg(\sum_{i=1}^4 xy^{2i}+ 2\sum_{i=1}^{3}\sum_{j = i+1}^{4} xy^{i+j}z\Bigg) + 1 \\
		&= F(x,y,z)\Bigg(x y^2 (1+ y^2) ((1+y^4) + 2 y z (1 +  y  +  y^2))\Bigg)+1.
\end{align*}
}
\section{Proof of Lemma \ref{lem:solveH_synthesis}}\label{appen_solve_synthesis}
We need to find the positive solution of
\begin{equation}
	\begin{split}
		H &= 0 \\
		xH_{x} & = \frac{yH_{y}}{\tau} = \frac{zH_{z}}{\delta} 
	\end{split}
\end{equation}
where $H = 1-xy^2(1+y^2)((1+y^4)+2yz(1+y+y^2))$.
\begin{proof}
	Take the partial derivatives as
	{\footnotesize
		\begin{align*}
			H &= 1-xy^2(1+y^2)((1+y^4)+2yz(1+y+y^2)),  \tag{i} \\
			xH_{x} &= -xy^2(1+y^2)((1+y^4)+2yz(1+y+y^2)), \tag{ii}\\
			\frac{zH_z}{\delta} &= \frac{-2xzy^3(1+y^2)(1+y+y^2)}{\delta}, \tag{iii}\\ 
			\frac{yH_y}{\tau} &= \frac{-2xy^2}{\tau}((1+2y^2)((1+y^4)+2yz(1+y+y^2)) \\								&\hspace{2mm}+(1+y^2)(2y^4+yz(1+2y+3y^2)))  \tag{iv}.
		\end{align*}
	}
	 Since $H = 0$, we get $xH_{x} = \frac{yH_{y}}{\tau} = \frac{zH_{z}}{\delta} = -1$. Further, equating (ii) and (iii) gives
	 \begin{equation}\label{eq:Zhat}
	   z = \frac{\delta(1+y^4)}{2(1-\delta)y(1+y+y^2)}.
	 \end{equation}
     We substitute $z$ in (ii) to obtain
     
      \begin{equation}\label{eq:Xhat}
        x =  \frac{1-\delta}{y^2(1+y^2)(1+y^4)}.
     \end{equation}	
	Then, we substitute $x$ and $z$ in (iv) to obtain
	{\small
	\begin{align}\label{eq:Yhat}
		& \frac{yH_{y}}{\tau} = \\ \notag & -\frac{
		2(1+y+y^2)(1+2y^2+3y^4+4y^6) + \delta (1-y^4)(1+y^2+y^4)}{\tau(1+y^2)(1+y^4)(1+y+y^2)} 
	\end{align}
	}
	Note that $\frac{yH_{y}}{\tau}= -1$.
Therefore, Lemma \ref{lem:solveH_synthesis} results from equations \ref{eq:Xhat}, \ref{eq:Zhat}, and \ref{eq:Yhat}.
\end{proof}

\section{Computation $\delta_{\max}$ and $y_{\min}$ in synthesis channel}\label{appen_solve_delta_syn}

Since evaluating the total ball size is the convex optimization, it can be proved easily that $\hat{z}$ is monotone increasing with $\delta$ (For example, see \cite{keshav2022evaluating} theorem 2) and further since $0 \le \hat{z} \le 1$, the maximum value of $\hat{z}$ is $1$. Hence for $\delta = \delta_{\max}$, we have that $\hat{z} = 1$. Further from computations, we observed that $\hat{x}$ and $\hat{y}$ decrease monotonically with $\delta$.

Therefore we have, $\frac{\delta_{\max}(1+y_{\min}^4)}{(1-\delta)(2y_{\min}(1+y_{\min}+y_{\min}^2))} = 1$ and hence 
\begin{equation}\label{eq:Delta_max}
	\delta_{\max} = \frac{2y_{\min}(1+y_{\min}+y_{\min}^2)}{(1+y_{\min}^4)+2y_{\min}(1+y_{\min}+y_{\min}^{2})}
\end{equation}

Substituting $\delta_{\max}$ in equation~\ref{eq:Yhat}, we obtain that $y_{\min}$ is the smallest positive real solution of the equation
{\footnotesize
	\begin{align*}
		\hspace{-2mm}
		\frac{y}{(1+y^4)+2y(1+y+y^2)} = \frac{\tau(1+y^2)(1+y^4)-(4y^6+3y^4+2y^2+1)}{(1-y^4)(y^4+2y^3+4y^2+2y+1)}.
	\end{align*}
}

}

\end{document}